\documentclass[12pt]{article}
\usepackage{bailey, sseq}
\pdfoutput=1

\begin{document}

\title{On the spectrum $\bo \sm \tmf$}
\author{Scott M. Bailey}
%\ead{bailey@math.rochester.edu}
%\address{University of Rochester, Department of Mathematics, R.C. Box 270138, Rochester, NY 14627-0138}
\maketitle

\begin{abstract}
M. Mahowald, in his work on $\bo$-resolutions, constructed a $\bo$-module splitting of the spectrum $\bo \sm \bo$ into a wedge of summands related to integral Brown-Gitler spectra.  In this paper, a similar splitting of $\bo \sm \tmf$ is constructed.  This splitting is then used to understand the $\bo_*$-algebra structure of $\bo_* \tmf$ and allows for a description of $\bo^* \tmf$. 
\end{abstract}
%
%\begin{keyword}
%connective $K$-theory \sep topological modular forms \sep integral Brown-Gitler spectra
%\end{keyword}

\section{Introduction}
All cohomology groups are assumed to have coefficients in $\FF_2$ and all spectra completed at the prime 2 unless stated otherwise.  Let $\SA$ denote the Steenrod algebra, and $\SA(n)$ the subalgebra generated by $\{ \Sq{1},\Sq{2},\cdots, \Sq{2^n}\}$.  Consider the Hopf algebra quotient $\SA//\SA(n) = \SA \otimes_{\SA(n)} \FF_2$.  Here the right action of $\SA(n)$ on $\SA$ is induced by the inclusion and the left action on $\FF_2$ by the augmentation.  Algebraically, one can consider the subsequent surjections
\[
	\SA \to \SA//\SA(0) \to \SA//\SA(1) \to \SA//\SA(2) \to \SA//\SA(3) \to \cdots
\]
and ask whether each algebra can be realized as the cohomology of some spectrum.  The case $n \geq 3$ requires the existance of a non-trivial map $S^{2^{n+1}-1} \to S^0$ which cannot occur due to Hopf invariant one.  For $n < 3$, however, it is now well-known that each algebra can indeed be realized by the cohomology of some spectrum:
\[
	H^* \es\FF_2 \to H^* \es\ZZ \to H^*\bo \to H^* \tmf
\]
There are maps realizing the above homomorphisms of cohomology groups
\[
\tmf \to \bo \to \es \ZZ \to \es \FF_2
\]
In particular, the spectrum $\tmf$ is at the top of a ``tower'' whose ``lower floors'' have been well studied in the literature, culminating with Mahowald's~\cite{Mah81} understanding of the spectrum $\bo \sm \bo$ and Carlsson's~\cite{Carls76} description of the cohomology operations $[\bo, \bo]$.  More difficult questions arise:  What is the structure of $\tmf \sm \tmf$?  What are the stable cohomology operations of $\tmf$?

We would like to understand the spectrum $\bo \sm \tmf$ for a variety of reasons.  First, it might serve as a nice intermediate step towards understanding the spectrum $\tmf \sm \tmf$.  Furthermore, determining its structure comes with an added bonus of understanding operations $[\tmf,\bo]$ which may provide some insight into understanding the cohomology operations of $\tmf$.  Second, the splitting of $\bo \sm \tmf$ has been instrumental to the author in demonstrating the splitting of the Tate spectrum of $\tmf$ into a wedge of suspensions of $\bo$.

Let $\BHZ{j}$ denote the $j^\textnormal{th}$ integral Brown-Gitler spectrum, whose homology will be described as a submodule of $H_* \es \ZZ$.  Such spectra have been studied extensively in the literature (see \cite{CDGM88},~\cite{GMJ86},~\cite{Shim84}, for example).  In particular, Mahowald~\cite{Mah81} demonstrated the splitting of $\bo$-module spectra $\bo \sm \bo \simeq \bigvee_{j \geq 0} \sus{4j} \bo \sm \BHZ{j}$.  Let $\Omega = \bigvee_{0 \leq j \leq i} \sus{8i + 4j} \BHZ{j}$.  The main theorem of this paper is the following
\begin{theorem}\label{t:main-theorem}
There is a homotopy equivalence of $\bo$-module spectra
\begin{equation}
\bo \sm \Omega \to \bo \sm \tmf
\end{equation}
\end{theorem}
The splitting is analogous to that of $\bo \sm \bo$ of Mahowald and even $M{\rm O}\langle 8 \rangle \sm \bo$ of Davis~\cite{Davis83}.  Its proof, therefore, contains ideas and results from both.  Section 2 deals with demonstrating an isomorphism on the level of homotopy groups, which first requires an understanding of the left $\SA(1)$-module structure of $H^* \tmf$.  In Section 3, we construct a map of $\bo$-module spectra realizing the isomorphism of homotopy groups.  Section 4 uses this splitting along with pairings of integral Brown-Gitler spectra to explicitly determine the $\bo_*$-algebra structure of $\bo_* \tmf$ and also identifies the cohomology $\bo^* \tmf$.

\section{The algebraic splitting}

The $E_2$-term of the Adams spectral sequence converging to the homotopy groups of $\bo \sm \tmf$ is given by
\begin{equation}
\EE_{\SA}^{s,t}\left( H^* (\bo \sm \tmf) , \FF_2 \right) \Rightarrow \pi_{t-s} (\bo \sm \tmf ).
\end{equation}
The $\EE$-group appearing in the above spectral sequence can be simplified via a change-of-rings isomorphism:
\begin{equation}\label{e:bostmfASScor}
\EE_{\SA(1)}^{s,t}\left( H^* \tmf, \FF_2 \right) \Rightarrow \pi_{t-s} (\bo \sm \tmf ).
\end{equation}
Therefore, it suffices to understand the left $\SA(1)$-module structure of $H^* \tmf$.  Computations and definitions simplify upon dualizing.  Indeed, the dual Steenrod algebra, $\SA_*$, is the graded polynomial ring $\FF_2[\xi_1,\xi_2,\xi_3,\hdots]$ with $\vert \xi_i \vert = 2^i - 1$.  An equivalent problem after dualizing is determining the right $\SA(1)$-module structure of the subring $H_* \tmf \subset \SA_*$.  The homology of $\tmf$ as a right $\SA$-module is given by Rezk~\cite{Rezk512}
\begin{equation}
H_{*} \tmf \cong \FF_{2}[ \zeta_{1}^{8}, \zeta_{2}^{4}, \zeta_{3}^{2}, \zeta_{4}, \hdots].
\end{equation}
The generators $\zeta_i = \chi \xi_i$, where $\chi: \SA_* \to \SA_*$ is the canonical antiautomorphism.  Define a new weight on elements of $\SA_*$ by $\omega(\zeta_i) = 2^{i-1}$ for $i\geq 1$.  For $a,b \in \SA_*$ define the weight on their product by $\omega(ab) = \omega(a) + \omega(b)$.  Let $N_k^\tmf$ denote the $\FF_2$-vector space inside $H_* \tmf$ generated by all monomials of weight $k$ with $N_0^\tmf = \FF_2$ generated by the identity.

\begin{lemma}\label{l:tmfA2}
As right $\SA(2)$-modules, 
\[
H_* \tmf \cong \bigoplus_{i\geq 0} N^\tmf_{8i}
\]
\end{lemma}

\begin{proof}
Certainly, the two modules are isomorphic as $\FF_2$-vector spaces.  To see there is an isomorphism of right $\SA(2)$-modules, note that the right action of the total square $\Sq{}=\sum_{i\geq 0} \Sq{i}$ on the generators of $H_* \tmf$ is given by:
\begin{align*}
\zeta_1^8 \cdot \Sq{} &= \zeta_1^8 + 1;\\
\zeta_2^4 \cdot \Sq{} &= \zeta_2^4 + \zeta_1^8 + 1;\\
\zeta_3^2 \cdot \Sq{} &= \zeta_3^2 + \zeta_2^4 + \zeta_1^8 + 1; \\
\zeta_n \cdot \Sq{} &= \sum_{i=0}^n \zeta_{n-i}^{2^i}
\end{align*}
for $n > 3$.  Since $\omega(1) = 0$, modulo the identity the total square preserves the weight of the generators of $H_* \tmf$.  Note that $\zeta_1^{2^{n-1}} \Sq{2^{n-1}} = 1$, hence the total square over $\SA(2)$ cannot contain a 1 in the expansion for dimensional reasons.
\end{proof}

Consider the homomorphism $V:\SA_* \to \SA_*$ defined on generators by
\[
V(\zeta_i) = 
\begin{cases}
1, &i = 0,1;\\
\zeta_{i-1}, &i \geq 2.\\
\end{cases}
\]
Restricting $V$ to the subring $H_* \tmf \subset \SA_*$ clearly provides a surjection $V_\tmf: H_* \tmf \to H_* \bo$.   Let $M_\bo (4i)$ denote the image of $N_{8i}^\tmf$ under the homomorphism $V_\tmf$. It is generated by all monomials with $\omega(\zeta^I) \leq 4i$.  The following proposition is clear.

\begin{proposition}\label{p:A2BG}
As right $\SA(2)$-modules
\begin{equation}
N_{8i}^\tmf \cong \sus{8i} M_\bo(4i).
\end{equation}
\end{proposition}
\begin{proof}
Due to the weight requirements, $V_\tmf$ is injective when restricted to $N_{8i}^\tmf$.  Indeed, the exponent of $\zeta_1$ in each monomial is uniquely determined by the other exponents.
\end{proof}

Additionally, if we denote by $N_k^\bo$ the $\FF_2$-vector space inside $H_* \bo$ generated by all elements of weight $k$ with $N_0^\bo = \FF_2$ generated by the identity, we have a similar lemma:
\begin{lemma}\label{l:MboA1}
As right $\SA(1)$-modules, 
\[
M_\bo(4i) \cong \bigoplus_{j=0}^n N_{4j}^\bo.
\]
\end{lemma}
Further restricting $V$ to the subring $H_* \bo$ provides a surjection $V_\bo : H_* \bo \to H_* \es \ZZ$.  Let $M_{\es \ZZ}(2j)$ denote the image of $N_{4j}^\bo$ under $V$.  This submodule is generated by all monomials with $\omega(\zeta^I) \leq 2j$.  As in Proposition~\ref{p:A2BG} we have the identification

\begin{proposition}\label{p:A1BG}
As right $\SA(1)$-modules,
\begin{equation}
N_{4j}^\bo \cong \sus{4j} M_{\es \ZZ}(2j).
\end{equation}
\end{proposition}

Goerss, Jones and Mahowald~\cite{GMJ86} identify the submodule $M_{\es \ZZ}(2j) \subset H_* \es \ZZ$ as the homology of the $j^{\txt{th}}$ integral Brown-Gitler spectrum:

\begin{theorem}[Goerss, Jones, Mahowald~\cite{GMJ86}]\label{t:BHZexist} 
For $j \geq 0$, there is a spectrum $\BHZ{j}$ and a map
\[
\BHZ{j} \xrightarrow{g} \es \ZZ
\]
such that
\begin{enumerate}
\item[(i)] $g_{*}$ sends $H_{*}(\BHZ{j})$ isomorphically onto the span of monomials of weight $\leq 2j$;
\item[(ii)] there are pairings
\[
\BHZ{m} \sm \BHZ{n} \to \BHZ{m+n}
\]
whose homology homomorphism is compatible with the multiplication in $H_{*}\es\ZZ$.
\end{enumerate}
\end{theorem}

\begin{remark}
The submodules $M_\bo(4i)$ are the so-called $\bo$-Brown-Gitler modules.  There is a family of spectra with similar properties, having these modules as their homology.  Proposition~\ref{p:A2BG} demonstrates that as an $\SA(2)$-module, $H_* \tmf$ is a direct sum of these modules.  On the level of spectra, however, $\tmf \sm \tmf$ does not split as a wedge of $\bo$-Brown-Gitler spectra.
\end{remark}

Combining the results of Lemmas~\ref{l:tmfA2} and~\ref{l:MboA1} with Theorem~\ref{t:BHZexist}, $H_* \tmf$ as a right $\SA(1)$-module can be written in terms of homology of integral Brown-Gitler spectra:
\begin{theorem}
As right $\SA(1)$-modules,
\[
H_{*}\tmf \cong \bigoplus_{0\leq j \leq i} \sus{8i+4j} H_* \BHZ{j}.
\]
\end{theorem}
The $E_2$-term of the Adams spectral sequence (\ref{e:bostmfASScor}) then becomes isomorphic to 
\begin{equation}
\bigoplus_{0 \leq j \leq i} \sus{8i+4j} \EE_{\SA(1)}^{s,t}\left( H^* \BHZ{j}, \FF_2 \right) \Rightarrow \pi_{t-s} (\bo \sm \tmf).
\end{equation}
This is precisely the Adams $E_2$-term converging to the homotopy of $\bo \sm \Omega$.  The chart can be obtained by applying the following theorem of Davis~\cite{DGM81} which links $\bo \sm \BHZ{n}$ to Adams covers of $\bo$ or $\bsp$, depending on the parity of $n$. 

%: B(\overline{n}) t:pairing
\begin{theorem}[Davis~\cite{DGM81}]\label{t:pairing}
If $\overline{n} = (n_{1},\hdots,n_{s})$, let $\vert \overline{n} \vert = \sum_{i=1}^{s} n_{i}$ and $\alpha(\overline{n}) = \sum_{i=1}^{s}\alpha(n_{i})$, and $\BHZ{\overline{n}} = \bigwedge_{i=1}^{s}\BHZ{n_{i}}$.  Then there are homotopy equivalences
\[
\bo \sm \BHZ{\overline{n}} \simeq K \vee
\begin{cases}
\bo^{2\vert \overline{n} \vert - \alpha(\overline{n})}, &\txt{if $\vert \overline{n} \vert$ is even;} \\
\bsp^{2\vert \overline{n} \vert - 1 - \alpha(\overline{n})}, &\txt{if $\vert \overline{n} \vert$ is odd;} \\
\end{cases}
\]
where $K$ is a wedge of suspensions of $\es\FF_{2}$.
\end{theorem}

\begin{figure}[htb]
\includegraphics[width=4.5in]{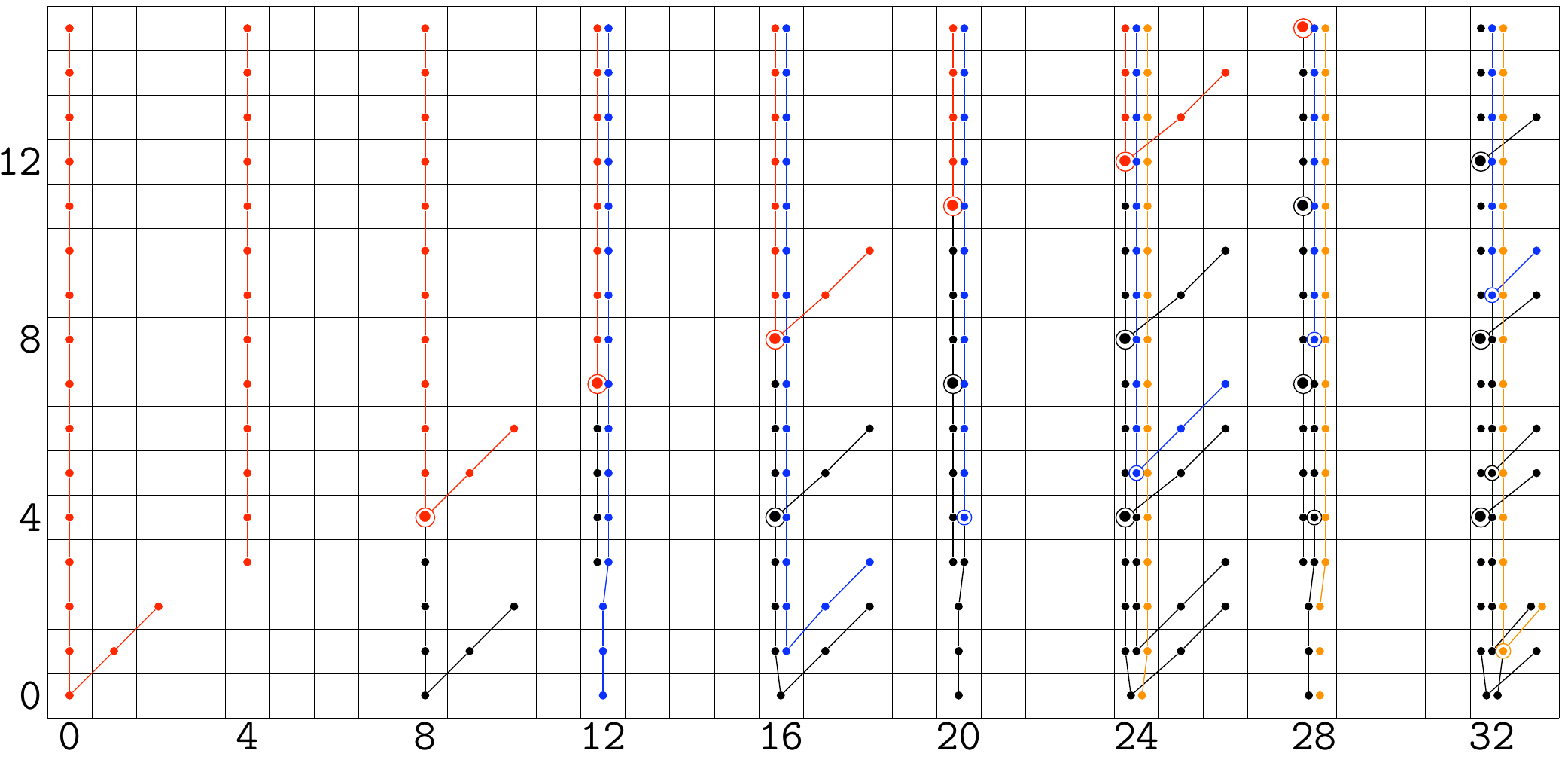}
\caption{$\EE^{s,t}_\SA \left(H^*(\bo \sm \tmf),\FF_2 \right) \Rightarrow \pi_{t-s}(\bo \sm \tmf)$}
\label{f:bostmf}
\end{figure}

The charts for $\bo$ and $\bsp$ are well known.  Using the above theorem along with the algebraic splitting of $H^* \tmf$, we see that Adams covers of $\bo$ begin in stems congruent to $0 \mod 8$ while Adams covers of $\bsp$ begin in stems congruent to $4 \mod 8$.  The first 32 stems of the chart for $\bo \sm \tmf$ is displayed in Figure~\ref{f:bostmf} modulo possible elements of order 2 in Adams filtration $s=0$ corresponding to free $\SA(1)$'s inside $H^* \tmf$.  The symbol $\bigodot$ appears in Figure~\ref{f:bostmf} to reduce clutter.  It is used to mark the beginning of another $\ZZ$-tower.  In general, all $\ZZ$-towers are found in stems congruent to $0 \mod 4$ while those supporting multiplication by $\eta$ occur in stems congruent to $4 \mod 8$.  

\begin{theorem}
There is an isomorphism of homotopy groups
\[
\pi_{*}(\bo \sm \tmf) \cong \pi_{*}\left( \bo \sm \Omega \right)
\] 
\end{theorem}

\begin{proof}
The $E_2$-terms of their respective Adams spectral sequences have been shown to be isomorphic.  Both spectral sequences collapse.  Indeed, the classes charted in Figure~\ref{f:bostmf} cannot support differentials for dimensional and naturality reasons.  Each element of order two in Adams filtration $s=0$ correspond to copies of $\SA(1)$ inside $H^* \tmf$.  These summands split off, obviating the existance of differentials.
\end{proof}

\section{The topological splitting} 
Theorem~\ref{t:main-theorem} concerns a $\bo$-module splitting of the spectrum $\bo \sm \tmf$.  The following observation will aid us in studying $\bo$-module maps.
\begin{lemma}\label{l:boequiv}
Let $X$ and $Y$ be spectra.  Then
	\[
		[\bo \sm X , \bo \sm Y]_\bo = [X, \bo \sm Y]
	\]
\end{lemma}
\begin{proof}
Let $\mathsf{u}_\bo:S^0 \to \bo$ and $\mathsf{m}_\bo: \bo \sm \bo \to \bo$ denote the unit and the product map of $\bo$, respectively.  Given $f: \bo \sm X \to \bo \sm Y$ and $g: X \to \bo \sm Y$, the equivalence is given by the composites
	\begin{align*}
	f &\mapsto f \circ ( \mathsf{u} \sm 1 ) \\
	g &\mapsto (\mathsf{m}_\bo \sm 1)\circ(1 \sm g)
	\end{align*} 
\end{proof}
The spectra $(\bo, \mathsf{m}_\bo, \mathsf{u}_\bo)$ and $(\tmf, \mathsf{m}_\tmf, \mathsf{u}_\tmf)$ are both unital $\mathcal{E}_\infty$-ring spectra~\cite{May77}.  This induces a unital $\mathcal{E}_\infty$-ring structure $(\bo \sm \tmf,\mathfrak{m},\mathfrak{u})$.  This structure will play an important role in the proof of the main theorem.  We begin by defining an increasing filtration of $\Omega$ via:
	\begin{equation}\label{e:filtration1}
		\Omega^n = \bigvee_{j=0}^n \bigvee_{i=j}^\infty \sus{8i+4j} \BHZ{j} 
	\end{equation}
Notationally, it will be convenient to let $B(j) = \sus{12j} \BHZ{j}$, so that the filtration (\ref{e:filtration1}) can be rewritten as
\begin{equation}\label{e:filtration2}
\Omega^n = \bigvee_{j = 0}^n \bigvee_{i\geq0} \sus{8i}B(j).
\end{equation}
The proof of Theorem~\ref{t:main-theorem} will proceed inductively on $n$.  We will assume the existance of a $\bo$-module map $\varrho_{2^i-1} : \bo \sm \Omega^{2^i-1} \to \bo \sm \tmf$ which is a stable $\SA$-isomorphism through a certain dimension.  The inductive step will be then to construct a $\bo$-module map $\varrho_{2^{i+1}-1} : \bo \sm \Omega^{2^{i+1}-1} \to \bo \sm \tmf$ which is a stable $\SA$-isomorphism through higher dimensions.  To do this, we will employ the pairings given in Theorem~\ref{t:BHZexist}(ii).  Define the map  
	\begin{equation}\label{e:intmap}
		g_{m,n}: \sus{8n} B(m) \to \bo \sm \tmf
	\end{equation}
to be the restriction of $\varrho_{2^i-1}$ to the summand $\sus{8n} B(m)$.  Denote by $g_m = g_{m,0}$.  

\begin{lemma}\label{l:extendmap}
	Let $m = 2^i$ and $0 \leq n < m$.  Suppose there are $\bo$-module maps $f_m: \bo \sm \BHZ{m} \to \bo \sm \tmf$ and $f_n : \bo \sm \BHZ{n} \to \bo \sm \tmf$ inducing injections on homology.  Then there is a $\bo$-module map
	\[
		f_{m+n}: \bo \sm \BHZ{m+n} \to \bo \sm \tmf
	\]
inducing an injection on homology.
\end{lemma}

\begin{proof}
	For all $0 \leq n < m$, Theorem~\ref{t:pairing} supplies equivalences of $\bo$-module spectra
		\begin{equation}
			\bo \sm \BHZ{m} \sm \BHZ{n} \simeq \left( \bo \sm \BHZ{m+n} \right) \vee K
		\end{equation}
	where $K$ is a wedge of suspensions of $\es \FF_2$.  There are no maps $[\es \FF_2, \bo \sm \tmf]$ so that the composite $\mathfrak{m} \circ(f_m \sm f_n)$ lifts as a $\bo$-module map to the first summand
		\[
			f_{m+n} : \bo \sm \BHZ{m+n} \to \bo \sm \tmf
		\]
hence is also an injection in homology.
\end{proof}

\begin{corollary}\label{c:extendmap}
	Suppose there are $\bo$-module maps $\varrho_{2^i-1}: \bo \sm \Omega^{2^i - 1} \to \bo \sm \tmf$ and $g_{2^i}: \bo \sm B(2^i) \to \bo \sm \tmf$ inducing injections on homology.  Then there is a $\bo$-module map 
		\[
			\varrho_{2^{i+1}-1} : \bo \sm \Omega^{2^{i+1}-1} \to \bo \sm \tmf
		\]
inducing an injection on homology groups.
\end{corollary}

\begin{proof}
For $0 \leq m \leq 2^i - 1$ and $n \geq 0$, there are $\bo$-module maps $g_{2^i+m,n}$ inducing an injection in homology.  These maps are obtained by applying Lemma~\ref{l:extendmap} to $g_{2^i}$ and the restriction of $\varrho_{2^i-1}$ to the summand
\[
g_{m,n} : \bo \sm \sus{8n} B(m) \to \bo \sm \tmf
\]
The map $\varrho_{2^{i+1}-1}$ is the wedge of these maps.
\end{proof}

The following observation will simplify our calculations inside the Adams spectral sequence.

%: extend 0 line l:homextendE2
\begin{lemma}\label{l:homextendE2}
Let $X$ and $Y$ be spectra.  Suppose $\mathcal{F}: \bo \sm X \to \bo \sm Y$ is given by the composite $(\mathsf{m}_\bo \sm Y)\circ(\bo \sm f)$ for some map $f: X \to \bo \sm Y$.  Then $\mathcal{F}_* (r x) = r \mathcal{F}_* (x)$ if $r \in \bo_*$ and $x \in \bo_* X$.
\end{lemma}

\begin{proof}
By construction, the composite $\mathcal{F}$ is a $\bo$-module map.
\end{proof}

In particular, the $\bo$-module map $\varrho_{2^{i+1}-1}$ constructed in Lemma~\ref{l:extendmap} induces a map in homotopy groups in Adams filtration $s=0$.  The above lemma allows us to apply the $\bo_*$-module structure to extend the morphism into positive Adams filtrations.  To complete the inductive step it suffices to construct a map $g_{2^i} : B(2^i) \to \bo \sm \tmf$ inducing an injection on homology.  Indeed, we can then apply Corollary~\ref{c:extendmap} to extend $\varrho_{2^i-1}$ to a $\bo$-module map
$\varrho_{2^{i+1}-1}: \bo \sm \Omega^{2^{i+1}-1} \to \bo \sm \tmf$.  

To construct  $g_{2^i}$, we will use $g_{2^{i-1}}$ supplied by the inductive hypothesis.  Once again we will attempt to use the pairing of integral Brown-Gitler spectra:
\begin{equation}\label{e:pairing2power}
\BHZ{2^{i-1}} \sm \BHZ{2^{i-1}} \to \BHZ{2^i}
\end{equation}
to construct a map $\bo \sm \BHZ{2^i} \to \bo \sm \tmf$.  Unfortunately, Lemma~\ref{l:extendmap} will not apply.  Indeed, the above pairings (\ref{e:pairing2power}) are not surjective in homology since the element corresponding to $\zeta_{i+3}$ inside $H_* \BHZ{2^i}$ is indecomposible.  To handle this case, we turn to a lemma of Mahowald~\cite{Mah81} made precise by Davis~\cite{Davis83}:
%: Davis stable A(1) fiber l:main-fiber
\begin{lemma}[Davis \cite{Davis83}]\label{l:main-fiber}
If $n$ is a power of $2$, let $F_n = \sus{8n-5} \MM \sm \BHZ{1}$.  There is a map $F_n \xrightarrow{j} \bo \sm \BHZ{n} \sm \BHZ{n}$ such that the cofibre of the composite
\[
\delta:\bo \sm F_n \xrightarrow{1 \sm j} \bo \sm \bo \sm \BHZ{n} \sm \BHZ{n} \xrightarrow {\mathsf{m}_\bo \sm 1 \sm 1} \bo \sm \BHZ{n} \sm \BHZ{n}
\]
is equivalent modulo suspensions of $\es \FF_2$ to $\bo \sm \BHZ{2n}$.
\end{lemma}

Define $\mathfrak{m}_{i-1}: \bo \sm B(2^{i-1}) \sm B(2^{i-1})  \to \bo \sm \tmf$ to be the $\bo$-module map induced by the composite $\mathfrak{m} \circ \left( g_{2^{i-1}} \sm g_{2^{i-1}}\right)$.   With Lemma~\ref{l:main-fiber} in mind, consider the diagram:
\[
\xymatrix{
\bo \sm \sus{2^{i+4}-5} \MM \sm \BHZ{1} \ar[r]^{\delta} &\bo \sm B(2^{i-1}) \sm B(2^{i-1}) \ar[d]_{\mathfrak{m}_{i-1}} \ar[r]&\bo \sm B(2^i) \ar@{-->}[dl]^{g_{2^i}}\\
&\bo \sm \tmf &
}
\]
It suffices to show the composite $\mathfrak{m}_{i-1} \delta$ is nulhomotopic, since then $\mathfrak{m}_{i-1}$ would then extend to the desired map $g_{2^i}$.  The following theorem is essentially due to Davis \cite[Prop. 2.8]{Davis83}, however modified to our context.

\begin{theorem}[Davis, \cite{Davis83}] \label{t:main-reduction}
Suppose $g_{2^{i-1}} : \bo \sm B(2^{i-1}) \to \bo \sm \tmf$ induces an injection on homology.  Then
\begin{equation}
\pi_{2^{i+4}-4}\left( \bo \sm B(2^{i-1})\sm B(2^{i-1})\right) \cong \ZZ_{(2)}
\end{equation}
with generator $\alpha_{2^{i+4}-4}$ whose image under $(\mathfrak{m}_{i-1})_\sharp$ is divisible by $2$.
\end{theorem}

\begin{proof}
Since $\bo \sm \tmf$ has the structure of an $\mathcal{E}_\infty$-ring spectrum, the map $\mathfrak{m}_{i-1}$ factors through the quadratic construction on $B(2^{i-1})$, i.e., there is a map $j$ making the the following diagram commute:
\[
\xymatrix{
& \bo \sm D_2(B(2^{i-1})) \ar[d]\\
\bo \sm B(2^{i-1}) \sm B(2^{i-1}) \ar[ur]^{j} \ar[r]_<<<<<<<<<<{\mathfrak{m}_{i-1}} &\bo \sm \tmf 
}
\]
where  
\[
D_2(B(2^{i-1})) = S^1 \ltimes_{\Sigma_2} (B(2^{i-1}) \sm B(2^{i-1}) ).
\]
Here the $\Sigma_2$-action on $S^1$ is the antipode and the action on the smash product interchanges factors.  Using this factorization, it suffices to show that the induced map $j_\sharp$ in homotopy sends the generator in dimension $2^{i+4}-4$ to twice an element of the homotopy of the quadratic construction.  This is proved by Davis~\cite{Davis83}.  
\end{proof}

\begin{proof}[Proof that Theorem~\ref{t:main-reduction} implies Theorem~\ref{t:main-theorem}]
Let $[x_i] \in \pi_i (\bo \sm \tmf)$ for $i = 0,8,12$ denote the classes in bidegree $(i,0)$ in the $E_2$-term displayed in Figure~\ref{f:bostmf}.  The class $[x_{12}]$ does not support action by $\eta$ so that $x_{12}$ extends to a map $B(1) \to \bo \sm \tmf$.   Upon smashing with $\bo$, we get maps
\begin{align*}
g_0 : \bo \sm B(0) &\to \bo \sm \tmf \\
g_{0,1} : \sus{8} \bo \sm B(0) &\to \bo \sm \tmf \\
g_{1}:  \bo \sm B(1) &\to \bo \sm \tmf 
\end{align*}
inducing injections in homology.   In particular, Lemma~\ref{l:extendmap} extends these to a $\bo$-module map $\varrho_1: \bo \sm \Omega^1 \to \bo \sm \tmf$ which is also an injection on homology.  Lemma~\ref{l:homextendE2} extends this morphism to positive Adams filtrations.  Figure~\ref{f:bostmf} demonstrates that modulo possible order 2 elements on the zero line, this map accounts for all homotopy classes through the 23-stem.  Hence, it is a stable $\SA$-equivalence in this range.

For the purpose of induction, assume the existance of a $\bo$-module map $\varrho_{2^i-1} : \bo \sm \Omega^{2^i - 1} \to \bo \sm \tmf$ inducing a stable $\SA$-equivalence through the $(12(2^i)-1)$-stem.  In particular, there is a map $g_{2^{i-1}} : \bo \sm B(2^{i-1}) \to \bo \sm \tmf$ of $\bo$-module spectra inducing an injection on homology groups.  Define $\mathfrak{m}_{i-1}$ and $\delta$ as above.  We will show $\mathfrak{m}_{i-1} \delta \simeq \ast$.

\begin{figure}[htb]
\[
\xygraph{
!{<0cm,0cm>;<1cm,0cm>:<0cm,1cm>::}
%Dimensions
!{(0,-1)}*{0} !{(1,-1)}*{1} !{(2,-1)}*{2} !{(3,-1)}*{3} !{(4,-1)}*{4} 
%Cells
!{(0,0)}*+{\circ}="c0"
!{(1,0)}*+{\circ}="c1"
!{(2,0)}*+{\circ}="c2"
!{(3,0)}*+{\circ}="c31" !{(3,1)}*+{\circ}="c32"
!{(4,1)}*+{\circ}="c4"
%Squares
"c0"-"c1" "c2"-"c31" "c32"-"c4"
"c0"-@/_1pc/"c2" "c1"-"c32" "c2"-"c4"
%!{(0,0)}-!{(6,0)}-!{(6,4)}-!{(0,0)}
}
\]
\caption{$H^*(\MM \sm \BHZ{1})$}
\label{f:MB1}
\end{figure}
Figure~\ref{f:MB1} shows the cell diagram for $H^* (\MM \sm \BHZ{1})$.  Since there are no elements of positive Adams filtration in stems congruent to $\{5,6,7\} \mod 8$ in the Adams spectral sequence converging to $\pi_* (\bo \sm \tmf)$, the composite $\mathfrak{m}_{i-1}\delta$ restricts to a map $\sus{2^{i+4}-5} \MM \to \bo \sm \tmf$.  Consider the composite
\[
S^{2^{i+4}-5} \xrightarrow{a_0} \sus{2^{i+4}-5} \MM \sm \BHZ{1} \xrightarrow{\mathfrak{m}_{i-1}\delta} \bo \sm \tmf
\]
restricting $\mathfrak{m}_{i-1}\delta$ to the bottom cell of $\sus{2^{i+4}-5} \MM$.  There are no elements of positive Adams filtration in stems congruent to $3 \mod 8$ so this restriction extends to the top cell
\[
S^{2^{i+4}-4}  \xrightarrow{a_1} \sus{2^{i+4}-5} \MM \sm \BHZ{1} \xrightarrow{\mathfrak{m}_{i-1}\delta} \bo \sm \tmf.
\]
Theorem~\ref{t:main-reduction} indicates that  the class $(\mathfrak{m}_{i-1})_\sharp(\delta a_1)$ is divisible by 2.  Hence, this map is nulhomotopic.  Applying Corollary~\ref{c:extendmap} gives the result.
\end{proof}

\section{The $\bo$-homology of  $\tmf$}

Both $\bo$ and $\tmf$ have the structure of $\mathcal{E}_\infty$-ring spectra, so that the smash product $\bo \sm \tmf$ also inherits such a structure.  The splitting of $\bo \sm \tmf$ into pieces involving integral Brown-Gitler spectra gives a nice description of its structure as a ring spectrum.  Indeed, the pairing of the $\BHZ{j}$ is compatible with multiplication inside $H_* \es \ZZ$ of which $H_* \tmf$ is a subring.  In particular, the pairings of the integral Brown-Gitler spectra induce the ring structure of $\bo \sm \tmf$.  The induced structure on homotopy groups is given by the following theorem:

\begin{theorem}
There is an isomorphism of graded $\bo_*$-algebras
\begin{equation}
\pi_* (\bo \sm \tmf) \cong \frac{\bo_* \left[\sigma, b_i, \mu_i \, \vert \, i \geq 0 \right]}{(\mu b_i^2 - 8 b_{i+1}, \mu b_i - 4\mu_i, \eta b_i )}\oplus F
\end{equation}
where $\vert \sigma \vert = 8$, $\vert b_i \vert = 2^{i+4}-4$, $\vert \mu_i \vert = 2^{i+4}$ and $F$ is a direct sum of $\FF_2$ in varying dimensions.
\end{theorem}

\begin{proof}
Theorem~\ref{t:pairing} gives homotopy equivalences
\[
\bo \sm B(n) \sm B(2^i) \to K \vee (\bo \sm B(n + 2^i)) 
\]
for all $n <  2^i$.  In particular, the induced pairings  
\[
\pi_* (\bo \sm B(n)) \otimes \pi_* ( \bo \sm B(2^i)) \to \pi_* (\bo \sm B(n+2^i))
\]
provide an isomorphism for all $n < 2^i$, modulo possible order 2 elements in Adams filtration zero corresponding to free $\SA(1)$ inside $H^* \tmf$.  Therefore, the homotopy classes inside $\bo$, $\sus{8}\bo$ and $\bo \sm B(2^i)$ for $i \geq 0$ generate the homotopy of $\pi_*(\bo \sm \tmf)$.  Hence, it suffices to examine the pairings 
\[
\bo \sm B(2^i) \sm B(2^i) \to \bo \sm B(2^{i+1}).
\]
\begin{figure}[h]
\[
\sseqxstart=12
\sseqystart=0
\sseqxstep=8
\sseqystep=4
\def\sseqgridstyle{\ssgridcrossword}
\def\sseqpacking{\sspackhorizontal}
\sseqentrysize=.4cm
\hfil
\begin{sseq}{26}{8}
\ssmoveto{12}{7} \ssbullstring{0}{-1}{8} \ssdroplabel[R]{b_0}
\ssmoveto{16}{7} \ssbullstring{0}{-1}{7} \ssdroplabel[R]{\mu_0} \e \e
\ssmoveto{20}{7} \ssbullstring{0}{-1}{4} 
\ssmoveto{24}{7} \ssbullstring{0}{-1}{3} \e \e
\end{sseq}
\hfil
\]
\caption{$\EE_{\SA(1)}^{s,t}(H^* B(1),\FF_2 )$}
\label{f:EB1}
\end{figure}
Figure~\ref{f:EB1} depicts the $E_2$-term of the Adams spectral sequence converging to $\bo_* B(1)$ along with its generators as a $\bo_*$-module.  With these generators, we can determine the decomposibles inside $\bo_* B(2^i)$.  Indeed, Lemma~\ref{l:main-fiber} provides us with a fiber sequence
\begin{equation}\label{e:decomp}
\bo \sm B(2^i) \sm B(2^i) \to \bo \sm B(2^{i+1}) \to \bo \sm \sus{2^{i+5}-4} \MM \sm \BHZ{1} 
\end{equation}
inducing a long exact sequence of $\EE$-groups.  Figure~\ref{f:EB2} shows how to use (\ref{e:decomp}) to form the $E_2$-page of $\bo \sm B(2^{i+1})$.  The arrows represent subsequent differentials and the dotted lines non-trivial extensions.  The classes in black are those contributed by $\bo \sm B(2^i) \sm B(2^i)$, i.e., the decomposible classes hit by multiplication by elements in the summand $\bo \sm B(2^i)$.  Those in red (or grey) are contributed by $ \bo \sm\sus{2^{i+5}-4} \MM \sm \BHZ{1}$.  Denote by $b_{i+1}$ the class found in bidegree $(2^{i+5}-4,0)$ and $\mu_{i+1}$ the class in $(2^{i+4}, 1)$.  These two elements are thus indecomposible in the ring $\pi_* ( \bo \sm \tmf)$.

Note that the class in bidegree $(2^{i+5}-8, 0)$ corresponds to the element $b_i^2$.  In particular, $\mu b_i^2 = 8 b_{i+1}$.  Also note that $\mu b_{i+1} = 4 \mu_{i+1}$ and $\eta b_{i+1} = 0$.

\begin{figure}[h]
\[
\sseqxstart=0
\sseqystart=0
\sseqxstep=0
\sseqystep=0
\def\sseqgridstyle{\ssgridcrossword}
\def\sseqpacking{\sspackhorizontal}
\sseqentrysize=.4cm
\hfil
\begin{sseq}{26}{8}
\ssmoveto{0}{0} \ssbullstring{0}{1}{8}
\ssmoveto{4}{0} \ssbullstring{0}{1}{8}
\ssmoveto{10}{0} \ssdropbull \ssdroplabel[R]{b_i^2} \bock \bock \bock \bock \bock \bock \bock
\ssmoveto{12}{0} \ssdropbull
\ssmoveto{14}{7} \ssbullstring{0}{-1}{7} \ssname{A1} 
\ssmoveto{18}{7} \ssbullstring{0}{-1}{6} \ssname{A2} \e \ssname{A3} \e \ssname{A4}
\ssmoveto{22}{7} \ssbullstring{0}{-1}{3} \ssname{A5}
%%%% Start MsmB1 %%%%
\renewcommand{\ssconncolor}{red} 
\renewcommand{\ssplacecolor}{red}
\ssmoveto{14}{0} \ssdropbull \ssdroplabel[R]{b_{i+1}}\ssname{B1}
\ssmoveto{18}{1} \ssdropbull \ssdroplabel[R]{\mu_{i+1}}\ssname{B2} \ssline{1}{1} \ssdropbull \ssline{1}{1} \ssdropbull \ssline{0}{-1} \ssdropbull \ssname{B3} \ssline{1}{1} \ssdropbull \ssname{B4} \ssline{1}{1} \ssdropbull \ssname{B5}
%%%% dashes %%%%
\renewcommand{\ssconncolor}{black}
\renewcommand{\ssplacecolor}{black}
\ssgoto{B1} \ssgoto{A1} \ssdashedstroke
\ssgoto{B2} \ssgoto{A2} \ssdashedstroke
\ssgoto{B5} \ssgoto{A5} \ssdashedstroke
%%%% arrows %%%%
\renewcommand{\ssconncolor}{blue}
\renewcommand{\ssplacecolor}{blue}
\ssgoto{B3} \ssgoto{A3} \ssstroke \ssarrowhead
\ssgoto{B4} \ssgoto{A4} \ssstroke \ssarrowhead
\end{sseq}
\hfil
\]
\caption{$\EE_{\SA(1)}^{s,t}(H^* B(2^{i+1}),\FF_2 )$}
\label{f:EB2}
\end{figure}
\end{proof}

\begin{remark}\label{r:bocotmf}
The splitting of $\bo \sm \tmf$ can also be used to give a description of the $\bo$-cohomology of $\tmf$.  Indeed, Lemma~\ref{l:boequiv} gives that $[\tmf , \bo] = [\bo \sm \tmf, \bo]_\bo$.  Since Theorem~\ref{t:main-theorem} provides a splitting as $\bo$-module spectra, one has the following chain of equivalences of $\bo^*$-comodules:
\begin{align*}
\bo^* \tmf &= [\tmf, \bo] \\
&= [\bo \sm \tmf, \bo]_{\bo} \\
&= \left[\bigvee_{m,n \geq 0} \sus{8n} \bo \sm B(m) , \bo \right]_\bo \\
&= \left[\bigvee_{m,n \geq 0}\sus{8n} B(m) , \bo \right] \\
&= \bigoplus_{m,n \geq 0} \sus{-8n}\bo^* B(m) 
\end{align*}
A complete description of the summands $\bo^* B(m)$ is given by Carlsson~\cite{Carls76}.  The comultiplication on $\bo^* \tmf$ is once again induced by the pairings of integral Brown-Gitler spectra.  It would be interesting to determine the explicit generators and relations as a $\bo^*$-coalgebra.
\end{remark}

\bibliographystyle{amsplain}
\bibliography{paper-list}

\end{document}